\documentclass[10pt,leqno]{article}
\usepackage{a4,latexsym,amssymb,amsfonts}
\usepackage{enumerate}
\usepackage{amsthm}
\usepackage{amsmath}
\usepackage[latin1]{inputenc}
\usepackage{eufrak}
\usepackage[usenames]{color}

\def\me{\mathsf{e}}
\def\mv{\mathsf{v}}
\def\Sin{\mathrm{Sin\,}}
\def\Cos{\mathrm{Cos\,}}
\def\Cot{\mathrm{Cot\,}}
\def\ea{\EuFrak{a}}

\def\ba{\mathbb{A}}

\newtheorem{theo}{Theorem}
\newtheorem{lemma}[theo]{Lemma}
\newtheorem{prop}[theo]{Proposition}
\newtheorem{cor}[theo]{Corollary}
\newtheorem{defi}[theo]{Definition}

\newtheorem{rem}[theo]{Remark}

\newtheorem{exa}[theo]{Example}

\numberwithin{equation}{section} \numberwithin{theo}{section}

\author{Marjeta Kramar Fijav\v{z}, Delio Mugnolo, and Eszter Sikolya}

\title{Variational and semigroup methods for waves and diffusion in networks}
\date{}
\begin{document}

\maketitle

\begin{abstract}
We study diffusion and wave equations in networks. Combining semigroup and variational methods we obtain well-posedness and many nice properties of the solutions in general $L^p$-context. Following earlier articles of other authors, we discuss how the spectrum of the generator can be connected to the structure of the network. We conclude by describing asymptotic behavior of solutions to the diffusion problem.
\end{abstract}

\section{Introduction}

In this paper we continue the study of dynamical processes in networks using semigroup methods. While~\cite{[KS05]} and~\cite{[MS06]} studied flow and
transport processes, the aim of the present paper is to combine variational and semigroup methods in order to obtain the well-posedness of initial value problems
associated with diffusion and wave equations. We thus consider
first and second order problems
$$\dot{u}_j(t,x)= (c_j u'_j)'(t,x)\quad {\rm and}\quad \ddot{u}_j(t,x)= (c_j u'_j)'(t,x),\qquad t\geq 0,\; x\in (0,1),$$
where $c_j(\cdot)$ and $u_j(t,\cdot)$ are functions on parameterized
edges $\me_j$ of a finite network. The node conditions in
\eqref{netcp2} and \eqref{netcp} below impose continuity and
Kirchhoff laws in the ramification vertices. Problems of this kind
have already been treated by many authors both from the mathematical and physical communities -- among others, we mention the earlier articles of
Lumer~\cite{[Lu80]}, Ali Mehmeti~\cite{[Al84]}, Roth~\cite{[Ro84]}, von Below~\cite{[Be85]}, Nicaise~\cite{[Ni85]}, Exner~\cite{[Ex89]},
Cattaneo~\cite{[Ca97]}, Kottos and Smilansky~\cite{[KS97]}, Kostrykin and Schrader~\cite{[KS99]}, and Kuchment~\cite{[Ku02]}, as well as the monographs~\cite{[Ni93]},~\cite{[Al94]}, and~\cite{[LLS94]}, and the proceedings~\cite{[ABN01]}.% We will incorporate and generalize most of these results.

Since the pioneering work of Beurling and Deny in the 1950s,
variational methods have been greatly developed. In combination
with the theory of strongly continuous semigroups of operators,
they provide a powerful tool to discuss properties of solutions to
many parabolic and hyperbolic problems, cf.~\cite{[Da89]},~\cite{[Ar04]}, and~\cite{[Ou04]}. While $L^2$-techniques
like the lemma of Lax--Milgram have been used in most of the above
mentioned papers, %in order to obtain well-posedness in $L^2$ only,
our paper seems to be the first applying variational methods to obtain
positivity, ultracontractivity, and stability for network equations
in a general $L^p$-context. This is the main aim of Sections~2
and~3. We remark that positivity of the semigroup governing the diffusion problem with much more general nodal conditions has been characterized, by algebraic methods, in~\cite{[KS06]}.

We then proceed to study the qualitative behavior of the solutions.
To that purpose we obtain in Section~4 a characteristic equation for
the spectrum of the generator and describe the appropriate
eigensolutions. We reprove some results
from~\cite{[Be85]},~\cite{[Ni85]},~\cite{[Ni87b]}, and~\cite{[Be88b]} in our setting
with slight generalizations. We see that the spectrum is determined by
the structure of the network and corresponds to the spectrum of the
Laplacian matrix known from graph theory, see~\cite{[Mo91]}. We give
an explicit connection between the two spectra and show the impact
of this to our problem. This relates to the well-known question:
``Can one hear the shape of a drum?", first addressed by Kac
in~\cite{[Ka66]}. Concerning differential operators on graphs, the
analogous question ``Can one her the shape of a network?" has been
formulated and answered in the negative by von Below in his
contribution to~\cite{[ABN01]}. Quite surprisingly, the same
question was raised at the same time in the almost homonymous
paper~\cite{[GS01]} by Guttkin and Smilansky. They answered it
\emph{in the positive}, by studying Schr{\"o}dinger operator on a
finite, simple graph \emph{with rationally independent arc lengths}
and imposing some further assumptions on matching conditions at the
vertices. In graph theory, however, it is well known that there are
many graphs sharing the same spectrum, see~\cite{[DH03]}. Also in
our case, the spectrum itself does not determine the network.

In Section~5 we study the asymptotic behavior of solutions to the
diffusion problem. To our knowledge this topic has not yet been
properly treated by other authors. We show that the solutions always
converge towards an equilibrium with rate of convergence depending
on the structure of the network. This is discussed for special
classes of networks. In similar contexts, convergence to equilibria has already been discussed, e.g., in~\cite{[BN96]}.

\section{The wave equation on a network}

We consider a finite connected network, represented by a
finite graph $G$ with $m$ edges $\me_1,\dots,\me_m$ and $n$
vertices $\mv_1,\dots,\mv_n$. We assume that all
the vertices have degree at least 2, i.e., that each vertex is
incident to at least 2 edges. Furthermore, we assume that $G$ is
\emph{simple}, that is it has no multiple edges or loops. We
normalize and parameterize the edges on the interval $[0,1]$. The
structure of the network is given by the $n\times m$ matrices
$\Phi^+:=(\phi^+_{ij})$ and $\Phi^-:=(\phi^-_{ij})$ defined by
\begin{equation*}
\phi^+_{ij}:\left\{
\begin{array}{rl}
1, & \hbox{if } \me_j(0)=\mv _i,\\
0, & \hbox{otherwise},
\end{array}
\right.
\qquad\hbox{and}\qquad
\phi^-_{ij}:\left\{
\begin{array}{rl}
1, & \hbox{if } \me_j(1)=\mv _i,\\
0, & \hbox{otherwise.}
\end{array}
\right.
\end{equation*}
We refer to~\cite{[KS05]} for terminology. The $n\times m$ matrix
$\Phi:=(\phi_{ij})$ defined by $$\Phi:=\Phi^+-\Phi^-$$ is known in
graph theory as \emph{incidence matrix} of the graph $G$. Further,
let $\Gamma(\mv_i)$ be the set of all the indices of the edges
having an endpoint at $\mv _i$, i.e.,
$$\Gamma(\mv _i):=\left\{j\in \{1,\ldots,m\}: \me_j(0)=\mv _i\hbox{ or } \me_j(1)=\mv _i\right\}.$$
For the sake of simplicity, we denote the value of the functions
$c_j(\cdot)$ and ${u}_j(t,\cdot)$ at $0$ or $1$ by $c_j(\mv_i)$ and
$u_j(t,\mv_i)$, if $\me_j(0)=\mv _i\hbox{ or } \me_j(1)=\mv _i$,
respectively. With an abuse of notation, we also set
$u'_j(t,\mv_i)=c_j(\mv_i):=0$ whenever $j\notin\Gamma(\mv_i)$. When
convenient, we shall also write the functions $u_j$ in vector form,
i.e., $u=(u_1,\dots,u_m)^\top$.

We start with the second order problem
\begin{equation}\label{netcp2}
\left\{\begin{array}{rclll}
\ddot{u}_j(t,x)&=& (c_j u_j')'(t,x), &t\in{\mathbb R},\; x\in(0,1),\; j=1,\dots,m, & (a)\\
u_j(t,\mv _i)&=&u_\ell (t,\mv _i), &t\in{\mathbb R},\; j,\ell\in \Gamma(\mv _i),\; i=1,\ldots,n,& (b)\\
\sum_{j=1}^m \phi_{ij}\mu_{j} c_j(\mv_i) u'_j(t,\mv_i)&=& 0, &t\in{\mathbb R},\; i=1,\ldots,n,& (c)\\
u_j(0,x)&=&\mathsf{f}_{j}(x), &x\in (0,1),\; j=1,\dots,m, & (d)\\
\dot{u}_j(0,x)&=&\mathsf{g}_{j}(x), &x\in(0,1),\; j=1,\dots,m,& (e)
\end{array}
\right.
\end{equation}
on the network. Note that $c_j(\cdot)$ and $u_j(t,\cdot)$ are
functions on the edge $\me_j$ of the network, so that the
right-hand side of~$(\ref{netcp2}a)$ reads in fact as
$$(c_j u_j')'(t,\cdot)=\frac{\partial}{\partial x}\left( c_j
\frac{\partial}{\partial x}u_j\right)(t,\cdot), \qquad t\in{\mathbb R},\; j=1,\ldots,m.$$

The functions $c_1,\ldots,c_m$ are the \emph{weights} of the edges,
and throughout this section we assume that $0<c_j\in H^1(0,1)$,
$j=1,\ldots,m$. They represent the different speeds of propagation
along each edge of the network $G$.  The
equation~$(\ref{netcp2}b)$ represents the continuity of the values
attained by the system at the vertices. The coefficients $\mu_j$,
$j=1,\ldots,m$, are strictly positive constants that influence the distribution
of impulse happening in the ramification nodes according to the Kirchhoff-type law~$(\ref{netcp2}c)$.

We now introduce \emph{weighted incidence matrices}
 $\Phi^+_w:=(\omega^+_{ij})$ and $\Phi^-_w:=(\omega^-_{ij})$ with entries
\begin{equation*}
\omega^+_{ij}:=\left\{
\begin{array}{ll}
\mu_j c_j(\mv_i), & \hbox{if } \me_j(0)=\mv _i,\\
0, & \hbox{otherwise},
\end{array}
\right. \qquad\hbox{and}\qquad \omega^-_{ij}:=\left\{
\begin{array}{ll}
\mu_j c_j(\mv_i), & \hbox{if } \me_j(1)=\mv _i,\\
0, & \hbox{otherwise}.
\end{array}
\right.
\end{equation*}

With these notations, equation $(\ref{netcp2}b)$ can be rewritten as
\begin{equation}\label{cont}
\exists d\in {\mathbb C}^n \hbox{ \rm{s.t.} } (\Phi^+)^\top d=u(t,0)\hbox{ \rm{and} }(\Phi^-)^\top d=u(t,1),\qquad
t\in{\mathbb R},
\end{equation}
while the Kirchhoff law $(\ref{netcp2}c)$ becomes
\begin{equation*}
\label{kir}
\Phi_w^+ u'(t,0)=\Phi_w^- u'(t,1), \qquad t\in{\mathbb R}.
\end{equation*}

%\bigskip
We are now in the position to rewrite our system in form of a
second order abstract Cauchy problem. First we consider the
(complex) Hilbert space
$$X_2:=\prod_{j=1}^m L^2(0,1; \mu_j dx)$$
endowed with the natural inner product
$$(f,g)_{X_2}:= \sum_{j=1}^m \int_0^1 f_j(x)\overline{g_j(x)} \mu_j dx,\qquad
f=\left(\begin{smallmatrix}f_1\\ \vdots\\
f_m\end{smallmatrix}\right),\;g=\left(\begin{smallmatrix}g_1\\ \vdots\\
g_m\end{smallmatrix}\right)\in X_2.$$

Observe that%, since the coefficients $\mu_j>0$ are fixed,
$X_2$ is isomorphic to $\left( L^2(0,1)\right)^m$ with equivalence of norms.
Moreover, $X_2$ is in fact a Hilbert lattice whose positive cone
consists of $m$ copies of the positive cone of $L^2(0,1; \mu_j dx)\approx L^2(0,1)$.
On $X_2$ we define an operator
\begin{equation}\label{amain}
A:=\begin{pmatrix}
\frac{d}{dx}\left(c_1 \frac{d}{dx}\right) & & 0\\
 & \ddots &\\
0 & & \frac{d}{dx}\left(c_m \frac{d}{dx}\right) \\
\end{pmatrix}
\end{equation}
with domain
\begin{equation}\label{domamain}
\begin{array}{rl}
D(A):=&\left\{f\in \left(H^2(0,1)\right)^m: \Phi_w^+ f'(0)=\Phi_w^- f'(1)\hbox{ and }\right.\\[0.3em]
&\left.\;\;\;\exists d\in {\mathbb C}^n \hbox{ s.t. } (\Phi^+)^\top d=f(0)\hbox{ and }(\Phi^-)^\top d=f(1)
\right\}.
\end{array}
\end{equation}

With this notations, we can finally rewrite~\eqref{netcp2} in form of a second order abstract Cauchy problem
\begin{equation}\label{acp2}
\left\{\begin{array}{rcll}
\ddot{u}(t)&=& Au(t), &t\in{\mathbb R},\\
u(0)&=&\mathsf{f},\\
\dot{u}(0)&=&\mathsf{g},
\end{array}
\right.
\end{equation}
on $X_2$. By means of variational techniques, we are going to show that $A$ enjoys several nice properties. We follow the techniques of~\cite{[Da89]} and~\cite[Sec.~7.1]{[ABHN01]}.

\begin{lemma}\label{main}
Consider the sesquilinear form
\begin{equation*}
\ea(f,g):=\sum_{j=1}^m\int_0^1 \mu_j c_j(x) f'_j(x) \overline{g'_j(x)} dx
\end{equation*}
on the Hilbert space $X_2$ with domain
\begin{equation*}
D\left(\ea\right)=V:=\left\{f\in \left(H^1(0,1)\right)^m:\exists d\in {\mathbb C}^n \hbox{ s.t. } (\Phi^+)^\top
d=f(0)\hbox{ and } (\Phi^-)^\top d=f(1)\right\}.
\end{equation*}
Then $\ea$ is densely defined and has the following properties:
\begin{itemize}
\item ({\sl symmetry}) : $\ea(f,g)=\overline{\ea(g,f)}$ for all  $f,g \in D(\ea)$,
\item ({\sl positivity}) : $\ea(f,f)\geq 0$ for all  $f \in D(\ea)$,
\item ({\sl closedness}) : $V$ is complete for the form norm $\Vert f\Vert_{\ea}:= \sqrt{\ea(f,f)+\Vert f\Vert^2_{X_2}}$,
\item({\sl continuity}) : $\vert \ea(f,g)\vert \leq  M\Vert f\Vert_\ea \Vert g\Vert_\ea$
for some $M>0$ and all $f,g\in D(\ea)$.
\end{itemize}
\end{lemma}

\begin{proof}
It is apparent that $V$ is a linear subspace of $X_2$. Observe that $\left(C^\infty_c(0,1)\right)^m\subset V$. It follows that $V$ is
dense in $X_2$, as by definition $L^2(0,1)$ is the closure of $C^\infty_c(0,1)$ in the $L^2$-norm.
By assumption, the weights $c_j$ are strictly positive, so that in particular $\ea$ is symmetric and also positive, since
\begin{equation*}
\ea(f,f) = \sum_{j=1}^m \int_0^1 \mu_j c_j(x)\vert f'_j(x)\vert^2 dx \geq 0\qquad\hbox{ for all } f\in V.
\end{equation*}

Furthermore, $V$ becomes a Hilbert space whenever equipped with the inner product
\begin{equation*}
(f,g)_V:=\sum_{j=1}^m\int_0^1 \left(f'_j(x)\overline{g'_j(x)} +
f_j(x)\overline{g_j(x)}\right) \mu_j dx,\qquad f,g\in V,
\end{equation*}
since $V$ is a closed subspace of $\big(H^1(0,1)\big)^m$. Set
$$c:=\min_{1\leq j\leq m}\min_{x\in[0,1]}c_j(x),\qquad
C:=\max_{1\leq j\leq m}\max_{x\in[0,1]}c_j(x).$$
Then one has
$$(c\wedge 1)\Vert f\Vert_V^2\leq \Vert f\Vert_{\ea}^2\leq (C\vee 1)\Vert f\Vert_V^2,\qquad f\in V,$$
so that the form norm $\Vert\cdot\Vert_\ea$ is equivalent to the
 norm $\Vert\cdot\Vert_V$. Since $V$ is complete with respect to $\Vert\cdot\Vert_V$, the closedness of $\ea$ follows at once.

Finally, $\ea$ is continuous. To see this, take $f,g\in V$ and observe that
\begin{eqnarray*}
\vert \ea(f,g)\vert&\leq& C\sum_{j=1}^m \vert \int_0^1 \mu_j f'_j(x) g'_j(x) dx\vert\\[0.3em]
&\leq& C\sum_{j=1}^m \Vert f'_j\Vert_{L^2(0,1; \mu_j dx)} \Vert g'_j\Vert_{L^2(0,1; \mu_j dx)}\\
&\leq& \frac{C}{2}
\left(\sum_{j=1}^m\Vert f'_j\Vert^2_{L^2(0,1; \mu_j dx)}\right)^\frac{1}{2}
\left(\sum_{j=1}^m\Vert g'_j\Vert^2_{L^2(0,1; \mu_j dx)}\right)^\frac{1}{2}\\
&\leq& \frac{C}{2\cdot (c\wedge 1)}\Vert f\Vert_\ea \Vert g\Vert_\ea,
\end{eqnarray*}
by the Cauchy--Schwartz inequality.
\end{proof}

\begin{defi}
From the form $\ea$ we can obtain a unique operator
$\left(B,D(B)\right)$ in the following way:
$$\begin{array}{rcl}
D(B)&:=&\left\{f\in V:\exists g\in X_2 \hbox{ s.t. } \ea(f,h)=(g,h)_{X_2}\; \forall h\in V\right\},\\
Bf&:=&-g.
\end{array}$$
We say that the operator $\left(B,D(B)\right)$ is \emph{associated
with the form $\ea$}.
\end{defi}

\begin{lemma}\label{ident}
The operator associated with the form $\ea$ is $\left(A,D(A)\right)$ defined in \eqref{amain}--\eqref{domamain}.
\end{lemma}

\begin{proof}
Denote by $\left(B,D(B)\right)$ the operator associated with $\ea$.
Let us first show that $A\subset B$. Take $f\in D(A)$. Then for all
$h\in V$
\begin{equation}\label{parts}
\begin{array}{rcl}
\ea(f,h)&=&\sum_{j=1}^m \int_0^1 \mu_j c_j(x) f'_j(x)\overline{h'_j(x)}dx\\[0.3em]
&=& \sum_{j=1}^m\left[\mu_j c_jf'_j\overline{h_j}\right]_0^1 - \sum_{j=1}^m \int_0^1 \mu_j (c_j f_j')'(x)\overline{h_j(x)}dx. % \\[0.3em]
\end{array}
\end{equation}

Using the incidence matrix $\Phi=\Phi^+ - \Phi^- $, the first term
above can be written as
$$\sum_{j=1}^m\left[\mu_j c_jf'_j\overline{h_j}\right]_0^1 = \sum_{j=1}^m\sum_{i=1}^n \mu_j c_j(\mv_i)(\phi_{ij}^- -\phi_{ij}^+)f'_j(\mv _i)\overline{h_j(\mv _i)}.$$
Observe now that the condition
$$\exists d\in {\mathbb C}^n\hbox{ s.t. } (\Phi^+)^\top d=h(0)\hbox{ and }(\Phi^-)^\top d=h(1)$$
in the definition of $V$ implies that $h$ is continuous in the
vertices, i.e., there exist $d_1,\ldots,d_n\in {\mathbb C}$ such
that $h_j(\mv _i)= d_i$ for all $j\in \Gamma(\mv _i)$,
$i=1,\ldots,n$. Summing up and using the other condition $\Phi^+_w
f'(0)=\Phi^-_w f'(1)$ in $D(A)$ we obtain that
$$\begin{array}{rcl}
\ea(f,h)&=& \sum_{i=1}^n \overline{d_i} \underbrace{\sum_{j=1}^m (\omega_{ij}^- -\omega_{ij}^+)f'_j(\mv _i)}_{=0} - \sum_{j=1}^m \int_0^1 (c_jf_j')'(x)\overline{h_j(x)}\mu_j dx\\
&=&-(Af,h)_{X_2},
\end{array}$$
which makes sense because $Af\in X_2$. The proof of the inclusion $A\subset B$ is completed.

To check the converse inclusion $B\subset A$ take $f\in D(B)$. By definition, there exists $g\in X_2$ such that
\begin{equation}\label{integr}
\sum_{j=1}^m \int_0^1 \mu_j c_j(x)f'_j(x)\overline{h'_j(x)}dx =\ea(f,h)=(g,h)_{X_2}= \sum_{j=1}^m \int_0^1 g_j(x)\overline{h_j(x)}\mu_j dx
\end{equation}
for all $h\in V$, hence in particular for all $h^j\in V$ of the form
\begin{equation*}
h^j=\left(\begin{smallmatrix}
0\\
\vdots\\
h_j\\
\vdots\\
0
\end{smallmatrix}\right)\leftarrow j^{\rm th}\hbox{ row},\;\; h_j\in H^1_0(0,1).
\end{equation*}
From this follows that~\eqref{integr} in fact implies
$$\int_0^1 \mu_j c_j(x)f'_j(x)\overline{h'_j(x)}dx = \int_0^1 g_j(x)\overline{h_j(x)}\mu_j dx\hbox{\; for all } j=1,\ldots,m,\;\; h_j\in H^1_0(0,1).$$
By definition of weak derivative this means that $c_j\cdot f_j'\in
H^1(0,1)$ for all $j=1,\ldots,m$. Since $0<c_j\in H^1(0,1)$,
it follows that in fact $f_j'\in H^1(0,1)$ for all
$j=1,\ldots,m$. We conclude that $f\in \left(H^2(0,1)\right)^m$.
Moreover, integrating by parts as in~\eqref{parts} we see that
if~\eqref{integr} holds for some $h\in V$, then
$$\sum_{i=1}^n {d_i} \sum_{j=1}^m (\omega_{ij}^- -\omega_{ij}^+)f'_j(\mv _i)=0,$$
where $d_i$ is the joint value attained at the vertex $\mv _i$ by
all $h_j$, $j\in\Gamma(\mv _i)$. Since $h\in V$ is arbitrary, this
means that
$$\sum_{j=1}^m (\omega_{ij}^- -\omega_{ij}^+)f'_j(\mv _i)=0\qquad \hbox{for all }i=1,\ldots,n,$$
that is, $\Phi_w^+ f'(0)=\Phi_w^- f'(1)$. Therefore $f\in D(A)$ and
\begin{equation*}
-\sum_{j=1}^m \int_0^1 \mu_j(c_jf'_j)'(x)\overline{h_j(x)}dx =\sum_{j=1}^m \int_0^1 g_j(x)\overline{h_j(x)}\mu_j dx
\end{equation*}
holds for all $h\in V$. This implies that $Af=-g$, and the proof is complete.
\end{proof}

We are now able to use some well-known results on sesquilinear
forms (cf.~\cite{[ABHN01],[Da89],[Ou04]}) in order to obtain nice
properties of our operator $A$.

\begin{prop}\label{generator}
The operator $\left(A,D(A)\right)$ defined
in~\eqref{amain}--\eqref{domamain} is self-adjoint and
dissipative. Thus, it generates a cosine operator function with
associated phase space $V\times X_2$. % and also a contractive analytic $C_0$-semigroup of angle $\frac{\pi}{2}$.
\end{prop}

\begin{proof}
By Lemmas \ref{main} and \ref{ident} we are in the situation
described in~\cite[Sec.~7.1]{[ABHN01]} for $H=X_2$, $V=D(\ea)$,
$(\cdot\vert\cdot)_V=\ea(\cdot ,\cdot)$, $\omega =1$, and $A=A_H$.
Thus the claim follows by~\cite[Proposition 7.1.1]{[ABHN01]},
\cite[Example 7.1.2]{[ABHN01]} and the fact that self-adjoint operators are unitarily equivalent to multiplication
operators (see also the remark at p.~413 in~\cite{[ABHN01]}).
\end{proof}

We can now state the main result of this section. This generalizes the well-posedness and regularity results
in~\cite{[Al84]},~\cite{[Al86]},~\cite{[Be88]}, and~\cite{[CF03]}, where only the case of constant or smooth coefficients $c_1,\ldots,c_m$ was considered

\begin{theo}~\label{wellp2}
The problem~\eqref{netcp2} is well-posed, i.e., for all
$\mathsf{f}\in V$ and $\mathsf{g}\in X_2$ it admits a unique
mild solution that continuously depends on the initial data.

If further $c_j\in C^\infty[0,1]$, $j=1,\ldots,m$, and the initial
conditions $\mathsf{f},\mathsf{g}\in
\left(C^\infty_c[0,1]\right)^m$, then the solution is of class
$\left(C^\infty[0,1]\right)^m$.
\end{theo}

\begin{proof}
It is well-known (see e.g.~\cite[Cor.~3.14.12]{[ABHN01]}) that
$$u(t):=C(t,A)\mathsf{f}+ S(t,A)\mathsf{g},\qquad t\in{\mathbb R},$$
yields the unique mild solution to~\eqref{acp2} for all initial
data $(\mathsf{f},\mathsf{g})$ in the phase space, where we denote by
$(C(t,A))_{t\in{\mathbb R}}$ and $(S(t,A))_{t\in{\mathbb R}}$
the cosine and sine operator functions generated by $A$,
respectively. %Moreover, it follows by~\cite[Thm.~3.14.11]{[ABHN01]} and~\cite[Prop.~II.5.2]{[EN00]} that $C(t,A)$ and $S(t,A)$ map $D(A^\infty)$ into itself for all $t\in{\mathbb R}$. Let now $c_1,\ldots c_m$ be of class $C^\infty[0,1]$. Then one sees that $\left(C^\infty_c((0,1))\right)^m\subset D(A^\infty)\subset \left(C^\infty[0,1]\right)^m$. Thus, we conclude that if the initial conditions $\mathsf{f},\mathsf{g}\in \left(C^\infty_c[0,1]\right)^m$, then the solution $u$ to the problem~\eqref{netcp2} is of class $\left(C^\infty[0,1]\right)^m$.
The assertion about regularity of solutions follows directly from basic properties of cosine and sine operator functions.
\end{proof}

\section{The heat equation on a network}

We now consider again the same network $G$ and, under the same assumptions and with the same notations of Section~2, we turn our
attention to the first order problem
\begin{equation}\label{netcp}
\left\{\begin{array}{rcll}
\dot{u}_j(t,x)&=& (c_j u_j')'(t,x), &t\geq 0,\; x\in(0,1),\; j=1,\dots,m,\\
u_j(t,\mv _i)&=&u_\ell (t,\mv _i), &t\geq 0,\; j,\ell\in \Gamma(\mv _i),\; i=1,\ldots,n,\\
\sum_{j=1}^m \mu_j\phi_{ij} c_j(\mv_i) u'_j(t,\mv_i)&=& 0, &t\geq 0,\; i=1,\ldots,n,\\
u_j(0,x)&=&{\mathsf{f}}_{j}(x), &x\in (0,1),\; j=1,\dots,m,\\
\end{array}
\right.
\end{equation}
This equation describes a diffusion process that takes place in a
network and $c_1,\ldots,c_m\in C^1[0,1]$ are (variable) diffusion coefficients
or conductances. Again, we are imposing continuity and
Kirchhoff-type conditions in the ramification nodes (controlled by
some constants $\mu_1,\ldots,\mu_m$).

It is already known that such a problem is well-posed in an $L^2$-context, cf.~\cite{[Be88]}. Moreover,  at least for the case of constant weights $c_1,\ldots,c_m$ and $\mu_1=\ldots=\mu_m=1$  the heat kernel has been computed in~\cite{[Ni87]}, thus yielding well-posedness in other $L^p$-spaces. We show by variational methods that the
semigroup governing~\eqref{netcp} is $L^\infty$-contractive, and hence we can extend the well-posedness result to an $L^p$-context by interpolation in the general case of variable diffusion coefficients. In particular, the analyticity of the $L^p$-semigroups seems to be a new result. Also observe that, by the bounded perturbation theorem, this also yields well-posedness for the Cauchy problem associated to the analogous cable equation, cf.~\cite{[Ni87b]}.

Let
$$X_p:=\prod_{j=1}^m L^p(0,1; \mu_j dx),\qquad p\in [1,\infty].$$
We have already seen in Proposition~\ref{generator} that $A$ is a
self-adjoint and dissipative operator on $X_2$. By the spectral theorem, this shows that $A$ generates a contractive, analytic semigroup of angle $\frac{\pi}{2}$, and in particular the first order abstract Cauchy problem
\begin{equation*}
\left\{\begin{array}{rcll}
\dot{u}(t)&=& Au(t), &t\geq 0,\\
u(0)&=&{\mathsf{f}},\\
\end{array}
\right.
\end{equation*}
is well-posed in $X_2$.  In fact, much more can be said.

\begin{lemma}\label{bd}
The semigroup $(T_2(t))_{t\geq 0}$ on $X_2$, associated with $\ea$,
is \emph{sub-Markovian}, i.e., it is real, positive, and contractive
on $X_\infty$.
\end{lemma}

\begin{proof}
By~\cite[Prop.~2.5, Thm.~2.7, and Cor.~2.17]{[Ou04]}, we need to
check that the following criteria are verified for the domain $V$
of $\ea$:
\begin{itemize}
\item ${f}\in V \Rightarrow \overline{f}\in V \hbox{ and } {\ea}({\rm Re}{f},{\rm Im}{f})\in\mathbb{R}$,
\item ${f}\in V, f\hbox{ real-valued }\Rightarrow \vert {f}\vert \in V \hbox{ and } \ea(\vert f\vert,\vert f\vert)\leq \ea(f,f)$,
\item $0\leq f\in V \Rightarrow 1\wedge {f}\in V \hbox{ and } \ea(1\wedge f,(f-1)^+)\geq 0$.
\end{itemize}

It is clear that $\overline{k}\in H^1(0,1)$ if $k\in H^1(0,1)$. Further, if $k$ is real valued, then $\vert
k\vert\in H^1(0,1)$ and $\vert k\vert'={\rm sign}k\cdot k'$, and if $0\leq k$, then $1\wedge k\in H^1(0,1)$ with
$(1\wedge k)'=k'{\mathbf{1}}_{\{k<1\}}$ and $((k-1)^{+})'=k'{\mathbf{1}}_{\{k>1\}}$.

By definition, the subspace $V$ contains exactly those functions
on the network that are continuous in the vertices (see
\eqref{cont}). Take any $f\in V$. By definition we have
$\overline{f_j}=(\overline{f})_j$, $1\leq j\leq m$. It follows
from the above arguments that $\overline{f}\in
\left(H^1(0,1)\right)^m$, and one can see that the continuity of
the values attained by $f$ in the vertices is preserved after
taking the complex conjugate $\overline{f}$. Hence,
$\overline{f}\in V$. Moreover, $\ea({\rm Re}{f},{\rm Im}{f})$ is
the sum of $m$ integrals. Recall that the weights are real-valued,
positive functions. Since all the integrated functions are
real-valued, it follows that ${\ea}({\rm Re}{f},{\rm
Im}{g})\in\mathbb R$. Thus, the first criterion has been checked.

Moreover, if $f$ is a real-valued function in $V$, then $\vert
f_j\vert =\vert f\vert_j$, $1\leq j\leq m$, and one sees as above
that $\vert f\vert\in V$. In particular, $\vert\vert
f\vert'\vert^2=\vert f'\vert^2$, and there holds
$$\ea(\vert f\vert,\vert f\vert)=\sum_{j=1}^m\int_0^1 \mu_j c_j(x)\vert f_j'(x)\vert^2 dx=\ea(f,f).$$
This shows that the second criterion applies.

Finally, take $0\leq f\in V$. Then
$$1\wedge f=1\wedge \left(\begin{smallmatrix}
f_1\\
\vdots\\
f_m\\
\end{smallmatrix}\right)
= \left(\begin{smallmatrix}
1\wedge f_1\\
\vdots\\
1\wedge f_m\\
\end{smallmatrix}\right),
$$
with all the functions $1\wedge f_j\in H^1(0,1)$, hence $1\wedge
f\in \left(H^1(0,1)\right)^m$. Again, the continuity of
$f$ in the vertices imposes the same property to the function
$1\wedge f$, i.e., $1\wedge f\in V$. Further, there holds
\begin{eqnarray*}
\ea(1\wedge f,(f-1)^+)&=&\sum_{j=1}^m\int_0^1 \mu_j c_j (1\wedge f_j)'(x) ((f_j-1)^+)'(x) dx\\[0.3em]
&=& \sum_{j=1}^m\int_0^1 \mu_j c_j f_j'(x) {\mathbf{1}}_{\{f_j<1\}}(x) f_j'(x) {\mathbf{1}}_{\{f_j>1\}}(x) dx=0.
\end{eqnarray*}
We have checked also the third criterion, thus the claim follows.
\end{proof}

\begin{lemma}\label{ultralemma}
The semigroup $(T_2(t))_{t\geq 0}$ on $X_2$ associated with $\ea$ is ultracontractive. In particular, it satisfies the estimate
\begin{equation}\label{ultra}
\Vert T_2(t)f\Vert_{X_\infty} \leq  M t^{-\frac{1}{4}}\Vert f\Vert_{X_2}
\qquad\hbox{ for all }t\in (0,1],\; f\in{X_2},
\end{equation}
for some constant $M$.
\end{lemma}

\begin{proof}
The form norm $\Vert\cdot\Vert_\ea$ on $V$ is equivalent to the norm $\Vert\cdot\Vert_V$, cf.~the proof of
Lemma~\ref{main}. Thus, by~\cite[Thm.~6.3 and following remark]{[Ou04]} it suffices to show that there holds
\begin{equation*}
\Vert f\Vert_{X_2}\leq M \Vert f\Vert_{V}^{\frac{1}{3}} \cdot \Vert f\Vert^{\frac{2}{3}}_{X_1}\qquad\hbox{for all } f\in V,
\end{equation*}
for some constant $M$. Recall the Nash inequality
\begin{equation}\label{nash}
\begin{array}{rcl}
\Vert k\Vert_{L^2(0,1)} &\leq& M_1 \left(\Vert k'\Vert_{L^2(0,1)}+\Vert k\Vert_{L^1(0,1)}\right)^{\frac{1}{3}} \cdot \Vert k\Vert_{L^1(0,1)}^{\frac{2}{3}}\\
&\leq & M_1 \Vert k\Vert_{H^1(0,1)}^{\frac{1}{3}} \cdot \Vert k\Vert_{L^1(0,1)}^{\frac{2}{3}},
\end{array}
\end{equation}
which is valid for all $k\in H^1(0,1)$ and some constant $M_1$, cf.~\cite[Thm.~1.4.8.1]{[Ma85]}.

Take finally $f\in V$ and observe that by~\eqref{nash}
\begin{eqnarray*}
\Vert f\Vert_{X_2}^2&=& \sum_{j=1}^m \Vert f_j\Vert_{L^2(0,1;\mu_j dx)}^2 \leq
M_1^2 \sum_{j=1}^m \Vert f_j\Vert_{H^1(0,1;\mu_j dx)}^{\frac{2}{3}} \cdot \Vert f_j\Vert_{L_1(0,1;\mu_j dx)}^{\frac{4}{3}}\\
&\leq & M_2 \left(\sum_{j=1}^m \Vert f_j\Vert_{H^1(0,1;\mu_j dx)}\right)^{\frac{2}{3}} \cdot \left(\sum_{j=1}^m\Vert f\Vert_{L^1(0,1;\mu_j dx)}\right)^{\frac{4}{3}}\\
&\leq & M_2 \Vert f\Vert_{V}^{\frac{2}{3}} \cdot \Vert f\Vert_{X_1}^{\frac{4}{3}},
\end{eqnarray*}
using the H\"{o}lder inequality. Thus, the claim follows.
\end{proof}

The following now holds by~\cite[Thm.~1.4.1, Thm.~1.6.4, and Thm.~2.1.5]{[Da89]} and~\cite[Thm.~3.13]{[Ou04]}.

\begin{cor}\label{comp}
The semigroup $(T_2(t))_{t\geq 0}$ extends to a family of compact,
contractive, positive one-parameter semigroups
$(T_p(t))_{t\geq 0}$ on $X_p$, $1\leq p\leq \infty$. Such
semigroups are strongly continuous if $p\in[1,\infty)$, and
analytic of angle $\frac{\pi}{2}-\arctan\frac{\vert
p-2\vert}{2\sqrt{p-1}}$ for $p\in(1,\infty)$.

Moreover, the spectrum of $A_p$ is independent of $p$, where  $A_p$
denotes the generator of $(T_p(t))_{t\geq 0}$, $1\leq p\leq \infty$.% All the eigenfunctions of $A=A_2$ are of class $X_\infty$.
\end{cor}

The estimate on the analyticity angle of $(T_p(t))_{t\geq 0}$ is not sharp, cf.~\cite{[Mu05]} for details.

\begin{rem}\label{integral}
\emph{Consider the part $\tilde{A}$ of $A$ in
$\left(C[0,1]\right)^m$, whose domain is given by
\begin{equation*}
\begin{array}{rl}
&D(\tilde{A})=\left\{f\in \left(C^2(0,1)\right)^m: \Phi_w^+ f'(0)=\Phi_w^- f'(1)\hbox{ and }\right.\\[0.3em]
&\left.\qquad\qquad\;\;\;\exists d\in {\mathbb C}^n \hbox{ s.t. }
(\Phi^+)^\top d=f(0)\hbox{ and }(\Phi^-)^\top d=f(1) \right\}.
\end{array}
\end{equation*}
Define
$$C(G):=\left\{f\in \left(C[0,1]\right)^m:
\exists d\in {\mathbb C}^n \hbox{ s.t. } (\Phi^+)^\top d=f(0)\hbox{
and }(\Phi^-)^\top d=f(1) \right\},$$ which can be looked at as the
space of all continuous functions on the graph $G$. It is easy to
see that $\overline{D(\tilde{A})}=C(G)$. By Corollary~\ref{comp}
$\tilde{A}$ has positive resolvent, and it follows
by~\cite[Thm.~3.11.9]{[ABHN01]} that its part in $C(G)$ generates a
positive strongly continuous semigroup.}
\end{rem}

In the next lemma we show that the generators of the semigroups in
the spaces $X_p,\, 1\leq p \leq \infty$ (see Corollary \ref{comp})
have in fact the same form as in $X_2$, with appropriate domain.

\begin{lemma}\label{identp}
For all $p\in [1,\infty]$ the generator $A_p$ of the semigroup
$(T_p(t))_{t\geq 0}$ is given by the operator defined in~\eqref{amain} with
domain
\begin{eqnarray*}
&D(A_p)=\left\{f\in \prod _{j=1}^m W^{2,p}(0,1;\mu_j dx): \Phi_w^+ f'(0)=\Phi_w^- f'(1)\hbox{ and }\right.\\
&\left.\qquad\qquad\;\;\;\exists d\in {\mathbb C}^n \hbox{ s.t. } (\Phi^+)^\top d=f(0)\hbox{ and }(\Phi^-)^\top d=f(1)
\right\}.
\end{eqnarray*}
\end{lemma}

\begin{proof}
Let us prove the claim for $p>2$. We have already remarked that $X_p\hookrightarrow X_q$ for all $1\leq
q\leq p\leq\infty$. Moreover, it follows by the ultracontractivity of $(T_2(t))_{t\geq 0}$ (see Lemma
\ref{ultralemma}) that $X_p$ is invariant under $(T_2(t))_{t\geq 0}$ for all $p>2$ because if $f\in X_p$ then
$f\in X_2$, and by \eqref{ultra}, $$\Vert T_2(t)f\Vert_{X_p}\leq C\cdot \Vert T_2(t)f\Vert_{X_\infty} \leq
C\cdot M t^{-\frac{1}{4}}\Vert f\Vert_{X_2} \leq  C'\cdot M t^{-\frac{1}{4}}\Vert f\Vert_{X_p}.$$ Thus,
by~\cite[Prop.~II.2.3]{[EN00]} the generator of $(T_p(t))_{t\geq 0}$ is the part of $A$ in $X_p$. A direct
computations yields the claim.

For $1\leq p<2$ the claim can be proven by duality, mimicking the proof of~\cite[Lemma~4.11]{[Mu05]}.
\end{proof}

\begin{theo}
The first order problem~\eqref{netcp} is well-posed on $X_p$,
$p\in[1,\infty)$, as well as on $C(G)$, i.e., for all initial data
$\mathsf{f}\in X_p$ or $\mathsf{f}\in C(G)$ the
problem~\eqref{netcp} admits a unique mild solution that
continuously depends on the initial data.

Such a solution is bounded in the time as well as (if $p>1$) in the space variables.
If further ${\mathsf f}\in X_p$, $1<p<\infty$, and $c_j\in C^\infty[0,1]$, $j=1,\ldots,m$, then the
solution $u(t,\cdot)$ is of class $\left(C^\infty[0,1]\right)^m$
for all $t>0$, and in particular the problem is solved pointwise for $t>0$.
\end{theo}

\begin{proof}
The well-posedness and boundedness results follow from the fact
that the operators $A_p$ generate ultracontractive analytic %and in fact ultracontractive
semigroups. If $c_j\in C^\infty[0,1]$, $j=1,\ldots,m$, then we can
show as in the proof of Theorem~\ref{wellp2} that
$D(A_p^\infty)\subset \left(C^\infty[0,1]\right)^m$ for all $p\in
[1,\infty]$. Since the semigroup $(T_p(t))_{t\geq 0}$ is analytic, $1<p<\infty$,
it maps $X_p$ into $D(A_p^\infty)$, and the claim follows.
\end{proof}

 \section{A characteristic equation}\label{secchareq}

Having proved that the Cauchy problem %s (\ref{netcp2}) and
(\ref{netcp}) is well-posed, we want to study the qualitative behavior
of its solutions. To this end we investigate the spectrum of the
generator $\left(A,D(A)\right)$. Since by Corollary \ref{comp} the
spectra of all $A_p$ on $X_p$, $1\leq p\leq\infty$, coincide, it
suffices to study the operator $A=A_2$ on $X_2$. Hence, we are interested in the spectrum of the operator
\begin{equation}\label{aspec}
A:=\begin{pmatrix}
c_1\frac{d^2}{dx^2} & & 0\\
 & \ddots &\\
0 & & c_m\frac{d^2}{dx^2} \\
\end{pmatrix}
\end{equation}
with domain
\begin{equation}\label{domaspec}
\begin{array}{rl}
D(A):=&\left\{f\in \prod _{j=1}^m \left(H^2(0,1);\mu_j dx\right):\Phi_w^+ f'(0)=\Phi_w^- f'(1)\hbox{ and }\right.\\[0.3em]
&\left.\;\;\;\exists d\in {\mathbb C}^n \hbox{ s.t. } (\Phi^+)^\top d=f(0)\hbox{ and }(\Phi^-)^\top d=f(1)
\right\}.
\end{array}
\end{equation}

Recalling properties of $\left(A,D(A)\right)$ already yields some information on its spectrum.

\begin{lemma}\label{sA}
The spectrum of $\left(A,D(A)\right)$ lies on the negative real line
and consists of eigenvalues only. Moreover, $s(A)=0\in\sigma(A).$
\end{lemma}

\begin{proof}
First note that $1\in D(A)$ and $A1=0$, thus $A$ is not invertible and $0\in\sigma(A)$. Since
$A$ generates a contractive semigroup (cf.~Proposition
\ref{generator}), $s(A)=0$. It follows from~\cite[II.4.30.(4)]{[EN00]} that, since $D(A)$ is contained in
$\left(H^2(0,1);\mu_j dx\right)^m$, the resolvent of $A$ is
compact. Therefore the operator $A$ only has point spectrum.
Recall that by Proposition \ref{generator} the operator $A$ is
self-adjoint, hence all its eigenvalues are real.
\end{proof}

From now on we
will assume that all the weights $c_i$, $i=1,\dots,m$, are
constant. Our aim is to find a `characteristic equation' for the
spectrum of $A$. In particular, we will be able to connect the
eigenvalues of the operator $A$ to the eigenvalues of the
\emph{Laplacian} or \emph{admittance matrix} of the corresponding
graph. This is the $n\times n$ matrix
\begin{equation}\label{graphlaplace}
\mathcal{L}:=D -\mathbb{A},
\end{equation}
where $\mathbb{A}$ is the standard $0-1$ adjacency matrix of the
graph and $D$ the diagonal matrix of vertex degrees. It is well
known that its spectrum reveals many properties of the graph, hence it is
used in many applications (see e.g.~\cite{[Ch97],[Me94],[Mo97]}).

We further define the \emph{generalized weighted adjacency matrix}
of the graph $G$ in the case $0< \lambda \neq c_j l^2\pi^2,\,
j=1,\dots, m, \, l\in \mathbb{Z}$, as
$$(\mathbb{A}_{C}(\lambda))_{ik}:=\left\{\begin{array}{ll}
0, & \text{ if } \Gamma(\mv _i)\cap \Gamma(\mv _k)=\emptyset,\\
\frac{\mu_j}{\sqrt{c_j}}\sin^{-1}\sqrt{\frac{\lambda}{c_j}},&\text{ if } j\in \Gamma(\mv _i)\cap \Gamma(\mv _k).
\end{array}
\right.
$$
By $D_{C}(\lambda)$ we denote the $n\times n$ diagonal matrix
(again for $0< \lambda \neq c_j l^2\pi^2,\, j=1,\dots, m, \, l\in
\mathbb{Z}$) defined as
$$D_{C}(\lambda):\mathrm{diag}\left(\sum_{j\in
\Gamma(\mv_i)}\frac{\mu_j}{\sqrt{c_j}}\cot\sqrt{\frac{\lambda}{c_j}}\right)_{i=1,\dots,n}.
%\end{array} \right.
$$ Finally, we define the \emph{generalized weighted
Laplacian matrix} as
$$\mathcal{L}_{C}(\lambda):=D_{C}\left(\lambda \right)-\mathbb{A}_{C}(\lambda).$$
We will now express the above matrices using the weighted incidence matrices.
For this purpose we define diagonal matrices
\begin{eqnarray*}
\Sin x&:=&\mathrm{diag}\left(\sin{\frac{x}{\sqrt{c_1}}},\dots,\sin{\frac{x}{\sqrt{c_m}}}\right),\\
\Cos x&:=&\mathrm{diag}\left(\cos{\frac{x}{\sqrt{c_1}}},\dots,\cos{\frac{x}{\sqrt{c_m}}}\right),\\
\Cot x &:=& \Sin^{-1}x\cdot \Cos x,\quad {\rm and}\\
C&:=&\mathrm{diag}(1/\sqrt{c_1},\dots,1/\sqrt{c_m}).
\end{eqnarray*}

\begin{lemma}\label{ad}
For $ 0< \lambda\ne c_rl^2\pi^2$, $r=1,\dots,m,\, l\in \mathbb{Z}$
we have
\begin{eqnarray*}
\mathbb{A}_{C}(\lambda)&=&
\Phi_w^+\cdot C\cdot\Sin^{-1}\sqrt{\lambda}\cdot(\Phi^-)^\top +
\Phi_w^-\cdot C\cdot\Sin^{-1}\sqrt{\lambda}\cdot(\Phi^+)^\top\quad {\rm and}\\
D_{C}(\lambda)&=&\Phi_w^+\cdot C\cdot\Cot\sqrt{\lambda}\cdot\left(\Phi^+\right)^\top + \Phi_w^-\cdot C\cdot\Cot\sqrt{\lambda}\cdot(\Phi^-)^\top.\\
\end{eqnarray*}\end{lemma}

We are now able to describe the spectrum of our operator $A$ in
terms of spectral values of $\mathcal{L}_{C}(\lambda)$. Similar
results have already been obtained in much the same way as
in~\cite{[Be85]},~\cite{[Ni85]},~\cite{[Ni87]},~\cite{[Ni87b]},~\cite{[Be88b]},
and~\cite{[Ca97]} for the cases $\mu_j=1$ and/or $c_j=1$.

\begin{theo}\label{spectrum1}
For the spectrum of the operator $(A,D(A))$, defined in \eqref{aspec}--\eqref{domaspec}, we obtain
$$\sigma(A)=\{0\}\cup\sigma_C\cup\sigma_{\cal L},$$
where
\begin{eqnarray*}
 \sigma_C &\subseteq &\left\{-c_ik^2\pi^2:  k\in\mathbb{Z}\setminus\{0\}, i=1,\dots,m\right\}\quad {\rm and} \\
  \sigma_{\cal L}&=& \left\{-\lambda\in\mathbb{R}_{-} : \lambda \neq c_i k^2 \pi^2,\,  \det\mathcal{L}_{C}(\lambda)=0\right\}.
\end{eqnarray*}
Furthermore,
\begin{enumerate}
\item $\lambda=0\in\sigma(A)$ is always an eigenvalue of
(geometric and algebraic) multiplicity $1$ with an eigenvector
$f(x)\equiv\mathbf{1}$, the constant $1$ function.
\item $-\lambda\in\sigma_{\cal L}$ is an eigenvalue of $A$ with
corresponding eigenvector
\begin{equation*}
f(x)=\Cos\sqrt{\lambda }x \cdot\left(\Phi^{+}\right)^{\top }\!\!d +
 \Sin^{-1}\sqrt{\lambda}\cdot \Sin\sqrt{\lambda}x\cdot\left(\left(\Phi^{-}\right)^{\top}\!\!  -
\Cos\sqrt{\lambda }\cdot\left(\Phi^+\right)^{\top}\right)d
\end{equation*}
where $d\in\ker\mathcal{L}_{C}(\lambda)$, and so the multiplicity $m(-\lambda)$ of this eigenvalue is equal to
$\dim\ker\mathcal{L}_{C}(\lambda)$;
\item $-c_ik^2\pi^2\in\sigma_C$ is an eigenvalue of $A$ if and only if
there exist $b\in\mathbb{C}^m$ and $d\in\mathbb{C}^n$ such that
whenever $j\in\Gamma(\mv_r)\cap\Gamma(\mv_s)$, $j\in\{1,\dots,m\}$,
we have
\begin{equation}\label{firstcond}
\left\{\begin{array}{ll}d_r=(-1)^{\sqrt{\frac{c_i}{c_j}}k}d_s,
&\text{ if
}\sqrt{\frac{c_i}{c_j}}k\in\mathbb{Z},\\
b_j = \sin^{-1}\sqrt{\frac{c_i}{c_j}}k\pi\cdot d_r
-\cot^{-1}\sqrt{\frac{c_i}{c_j}}k\pi\cdot d_s, &\text{ otherwise.}
\end{array}\right.
\end{equation}
These vectors further satisfy the equation
\begin{equation}\label{othercond}
\Phi_w^-\cdot C \cdot
\Sin\sqrt{c_i}k\pi\cdot(\Phi^+)^{\top}d=(\Phi_w^-\cdot C \cdot
\Cos\sqrt{c_i}k\pi-\Phi_w^+ \cdot C)\cdot b.
\end{equation}
If the eigenvector exists, then it has the form
\begin{equation*}
f(x)=\Cos\sqrt{c_i}k\pi x\cdot(\Phi^+)^\top d + \Sin\sqrt{c_i}k\pi
x\cdot b.
\end{equation*}
\end{enumerate}
\end{theo}

\begin{proof}
By Lemma \ref{sA}, we need to solve the equation
$$Af=-\lambda f\quad{\rm for}\quad
f\in D(A)\quad {\rm and}\quad\lambda \ge 0.$$

We will distinguish three cases.

\noindent{\bf Case 1:} Assume that $\lambda\ne c_i k^2\pi^2$ for all $k\in\mathbb{Z}, i=1,\dots,m$.\\
In this case the eigenfunctions of $A$ are of the form
$$f(x)=\Cos\sqrt{\lambda}x\cdot a + \Sin\sqrt{\lambda}x\cdot b \quad \hbox{\rm for some}\quad a,b\in \mathbb{C}^{m}.$$
From the continuity assumption in the domain of $A$ (see
(\ref{domaspec}))
$$\exists d\in {\mathbb C}^n\hbox{ s.t. } (\Phi^+)^\top\!\!d=f(0)\hbox{ and }(\Phi^-)^\top\!\! d=f(1)$$
we obtain
$$
f\left( x\right) =\Cos\sqrt{\lambda }x \cdot
\left(\Phi^{+}\right)^{\top }\!\!d +
 \Sin^{-1}\sqrt{\lambda}\cdot \Sin\sqrt{\lambda}x\cdot\left(\left(\Phi^{-}\right)^{\top}\!\!  -
\Cos\sqrt{\lambda }\cdot\left(\Phi^+\right)^{\top}\right)d
$$
for some $d\in \mathbb{C}^{n}$. The other condition $\Phi_w^+
f'(0)=\Phi_w^- f'(1)$ in the domain $D(A)$ (i.e.~the Kirchhoff
law) yields that $f\in \ker \left( \lambda -A\right) $ if and only
if the vector $d\in \mathbb{C}^{n}$ satisfies
\begin{eqnarray*}
&&\Phi _{w}^+\cdot C\cdot\Sin^{-1}\sqrt{\lambda} \cdot
\left(\left(\Phi^{-}\right)^{\top} - \Cos\sqrt{\lambda }\cdot\left(\Phi^+\right)^{\top}\right)d=\\
&&=\Phi_{w}^-\cdot C\cdot\left(\Cot\sqrt{\lambda}\cdot
\left((\Phi^{-})^{\top }-\Cos\sqrt{\lambda}\cdot\left(\Phi^+\right)^{\top}\right)-\Sin\sqrt{\lambda}\cdot
(\Phi^+)^{\top}\right)d.
\end{eqnarray*}
Observe now that, by Lemma \ref{ad}, the following two terms are
the previously defined diagonal matrix
$$\Phi_w^+\cdot C\cdot\Cot\sqrt{\lambda}\cdot\left(\Phi^+\right)^\top + \Phi_w^-\cdot C\cdot\Cot\sqrt{\lambda}\cdot(\Phi^-)^\top=D_{C}(\lambda),
$$
while rearranging the remaining terms yields the weighted adjacency
matrix
\begin{eqnarray*}
 && \Phi_w^+\cdot C\cdot\Sin^{-1}\sqrt{\lambda}\cdot(\Phi^-)^\top +
\Phi_w^-\cdot C\cdot\left(\Sin^{-1}\sqrt{\lambda}\cdot\Cos^2\sqrt{\lambda} + \Sin\sqrt{\lambda }\right)\cdot(\Phi^+)^\top = \\
  && =\Phi_w^+\cdot C\cdot\Sin^{-1}\sqrt{\lambda}\cdot(\Phi^-)^\top +
\Phi_w^-\cdot C\cdot\Sin^{-1}\sqrt{\lambda}\cdot(\Phi^+)^\top =\mathbb{A}_{C}(\lambda).
\end{eqnarray*}
Summing up, the condition for $d\in \mathbb{C}^{n}$ becomes
$$\left(\mathbb{A}_{C}(\lambda)-D_{C}(\lambda)\right)d=0,\quad \hbox{\rm that is}\quad d\in\ker\mathcal{L}_{C}(\lambda).$$

\smallskip
\noindent{\bf Case 2:} $\lambda=0$.\\
The eigenfunctions of $A$, corresponding to $\lambda=0$, are of the form
$$f(x)=x\cdot a + b \quad \hbox{\rm for some}\quad a,b\in \mathbb{C}^{m}.$$
We repeat the above procedure and the conditions in the domain (\ref{domamain}) yield
$$
f(x)=(\Phi^+)^\top d - x\cdot\Phi^\top d\quad {\rm for}\quad
d\in\ker\Phi_w\Phi^\top\quad {\rm
with}\quad\Phi_w=\Phi_w^{+}-\Phi_w^{-}.$$
Since our graph is
connected, the multiplicity of $0$ in $\sigma(\Phi_w\Phi^\top)$ is 1
(cf.~\cite[Lemma 1.7.(iv)]{[Ch97]} or~\cite[Prop.~2.3]{[Mo97]})). It
is easy to see that the corresponding eigenvector equals
$d=\mathbf{1}:=(1,\dots,1)^\top$. Now compute
$$f(x)=(\Phi^+)^\top \mathbf{1} - x\cdot\Phi^\top \mathbf{1} \equiv \mathbf{1}\text{ for all }x.$$

\smallskip
\noindent{\bf Case 3:} $\lambda =c_ik^2\pi^2$ for some $k\in\mathbb{Z}\setminus\{0\}$ and some $i\in\{1,\dots,m\}$. \\
We proceed as before while some care need to be taken with zero
entries that arise. Before applying the inverse of $\Sin \sqrt{\lambda}$, the continuity
condition in (\ref{domaspec}) implies
$$f(x)=\Cos\sqrt{c_i}k\pi x\cdot(\Phi^+)^\top d + \Sin\sqrt{c_i}k\pi x\cdot b,$$
where $b$ satisfies the equation
\begin{equation}\label{b}
\Sin\sqrt{c_i}k\pi\cdot b = \left((\Phi^-)^\top -
\Cos\sqrt{c_i}k\pi\cdot(\Phi^+)^\top\right) d.
\end{equation}
Since the $i$-th entry on the left-hand side equals $0$, the
vector $d$ should satisfy the condition
$$d_p=(-1)^kd_q\quad {\rm for}\quad i\in\Gamma(\mv_p)\cap\Gamma(\mv_q).$$
Moreover, if for any other $j\in\{1,\dots,m\}$ we have $\sqrt{\frac{c_i}{c_j}}k\in\mathbb{Z}$, then also
$$d_r=(-1)^{\sqrt{\frac{c_i}{c_j}}k}d_s\quad {\rm for}\quad j\in\Gamma(\mv_r)\cap\Gamma(\mv_s).$$
For each of these cases we have no conditions on $b_j$. If on the
other hand $\sqrt{\frac{c_i}{c_j}}k\notin\mathbb{Z}$, equation
(\ref{b}) yields
$$b_j = \sin^{-1}\sqrt{\frac{c_i}{c_j}}k\pi\cdot d_r
-\cot^{-1}\sqrt{\frac{c_i}{c_j}}k\pi\cdot d_s,\quad
j\in\Gamma(\mv_r)\cap\Gamma(\mv_s).$$ Furthermore, the condition
$\Phi_w^+ f'(0)=\Phi_w^- f'(1)$ in the domain $D(A)$ (i.e.~the
Kirchhoff law) implies that above vectors $d$  and $b$ have to
satisfy the equation (\ref{othercond}).
\end{proof}

Let us emphasize that the condition 3 in the above theorem in not
always satisfied, %-- it depends on the structure of the graph and on the weights $c_i$. The whole nonzero
therefore the spectrum of our operator $A$ strongly relies on the
underlying graph and on the weights $c_j$.

From now on we assume $c_j=1,\, j=1,\dots,m$. In this case we are
able to connect the spectrum of the operator $A$ to the spectrum of
yet another matrix known in graph theory. The \emph{transition
matrix} is defined as
$$\mathbb{P}:=D^{-1}\mathbb{A}$$
and is studied in connections with random walks on graphs (see
e.g.~\cite[Sec.~5.2]{[Mo97]} or~\cite[Sec.~1.5]{[Ch97]}). The matrix
$\mathbb{P}$ is always positive, symmetric, row stochastic matrix
with eigenvalues
$$\sigma \left( \mathbb{P}\right)  =\left\{ \alpha_{1},\dots,\alpha_{n}\right\} \text{ where }-1 \leq \alpha_{n}\leq \dots\leq \alpha _{2}<\alpha_{1}=1.$$
By~\cite[Claim 5.3]{[Mo97]}, $1$ is a simple eigenvalue whenever $G$
is connected (what we assumed at the beginning) and $-1$ is an
eigenvalue if and only if $G$ is bipartite (see also Lemma 1 in
Section $5$ of~\cite{[Be85]}). It turns out that an important subset
of the spectrum depends on the fact whether the graph $G$ is
\emph{bipartite} or not. This property means that the set of
vertices can be divided into two disjoint subsets $V_1$ and $V_2$
such that any edge of $G$ has one endpoint in one and the other
endpoint in the other subset. Note that $G$ is bipartite if and only
if it does not have any odd cycle.

The following characteristic equation has already been proved by von
Below~\cite{[Be85]}. We state and sketch the proof in our context
for the convenience of the reader.

\begin{theo}\label{spectrum2}
Let $(A,D(A))$ be the operator defined in
\eqref{aspec}--\eqref{domaspec}, with  $c_j=1,\, j=1,\dots,m.$
Then for the spectrum of $A$ we have
$$\sigma(A)=\{0\} \cup \sigma_{p}\cup\sigma_k,$$
where
\begin{equation*}
\sigma_p =\left\{-\left( 2l\pi \pm \mathrm{arc}\cos \alpha
\right)^{2}: \alpha\in\sigma(\mathbb{P})\setminus \{-1,1\} \text{
and }l\in\mathbb{Z} \right\} ,
\end{equation*}
and
\begin{equation*}
\sigma_k= \{-k^2\pi^2: k\in\mathbb{Z}\setminus \{0\}\}
\end{equation*}
For the multiplicities of the eigenvalues we have:
\begin{enumerate}
\item $m(0)=1$; \item $m(-\lambda)=\dim \ker \left( \cos\sqrt{\lambda }\cdot I-\mathbb{P}\right) $ for $-\lambda
\in \sigma_p$; \item $m(-k^2\pi^2)=m-n+2$, if $G$  is bipartite; \item $m(-4l^2\pi^2)=m-n+2$ and
$m(-(2l+1)^2\pi^2)=m-n$, if $G$ is not bipartite.
\end{enumerate}
\end{theo}

\begin{proof}
We use Theorem \ref{spectrum1} for $C=I$. For the spectral point
$\lambda=0$ the statement follows directly from Theorem
\ref{spectrum1}. Assume first that $\lambda\ne k^2\pi^2$ for any
$k\in\mathbb{Z}$. Then
\begin{equation*}
\mathcal{L}_{I}(\lambda)=\sin^{-1}\sqrt{\lambda}\left(\cos\sqrt{\lambda}\cdot
D-\mathbb{A}\right),
\end{equation*}
and the characteristic equation becomes
\begin{equation*}
 -\lambda \in \sigma \left( A\right) \Longleftrightarrow \det\left(\cos\sqrt{\lambda}\cdot D-\mathbb{A}\right)=0\Longleftrightarrow \det \left( \cos\sqrt{\lambda }\cdot I-\mathbb{P}\right) =0
\end{equation*}
for the transition matrix $D ^{-1}\mathbb{A}=\mathbb{P}$. The last equivalence says that
\begin{equation*}
-\lambda \in \sigma \left( A\right) \Longleftrightarrow
\cos\sqrt{\lambda } \in \sigma \left(
\mathbb{P}\right)\Longleftrightarrow \lambda =\left(
2l\pi\pm\mathrm{arc}\cos \alpha \right)^{2}
\end{equation*}
for some $\alpha \in \sigma \left(\mathbb{P}\right)$, $-1<\alpha <
1$, and $l\in\mathbb{Z}$. Since $\dim \ker (\cos\sqrt{\lambda}\cdot
D-\mathbb{A})=\dim \ker ( \cos\sqrt{\lambda }\cdot I-\mathbb{P})$,
statement 2 also follows by Theorem \ref{spectrum1}.2.

Now, let $\lambda= k^2\pi^2$ for some $k\in\mathbb{Z}\setminus
\{0\}$. Observe that condition \eqref{firstcond} in Theorem
\ref{spectrum1}.3 becomes
$$d_r=(-1)^k d_s\quad {\rm whenever}\quad\Gamma(\mv_r)\cap\Gamma(\mv_s)\ne\emptyset. $$
If $k$ is even this condition is always fulfilled for $d=c\cdot
\mathbf{1}$ for any $c\in \mathbb{R}$. Because the network is
assumed to be connected, there is no other solution. For odd $k$ we
can always choose $d=0$. However, we can find a nonzero $d$ only in
the case $G$ does not have any odd cycles, that is when $G$ is
bipartite -- hence, when his set of vertices can be divided into two
disjoint subsets $V_1$ and $V_2$ such that any edge of $G$ has one
endpoint in one and the other endpoint in the other subset. If this
holds, the coordinates of $d$ can be chosen in such a way that at
the places of vertex indices belonging to $V_1$ we set $c$ and at
the places of vertex indices belonging to $V_2$ we set $-c$,
$c\in\mathbb{R}$. By connectivity this are again all possible
solutions.

Since $\sin k\pi=0$ and $\cos k\pi=(-1)^k$, the other condition
\eqref{othercond} in Theorem \ref{spectrum1}.3 becomes
$$\left(\Phi_w^+ - (-1)^k \Phi_w^-\right)b=0.$$

Using the proof of~\cite[Theorem 5 (17)]{[Be85]} we obtain $3.$ and
$4.$
\end{proof}

\section{Stability results for the diffusion problem}

In the last section we are interested in the asymptotic behavior of
solutions to the problem \eqref{netcp}. By Corollary~\ref{comp}, the
corresponding semigroup $T_p(t)_{t\geq 0}$ on
$X_p=\left(L^p(0,1); \mu_j dx\right)^m$, $1\leq p<\infty$, has many nice
properties: it is contractive, compact, positive. These properties already yield
norm convergence of the solutions to an equilibrium
(cf.~\cite[Cor. V.2.15]{[EN00]}).

From the connectedness of our graph, used in the proof of Thm.
4.3.1, we obtain another useful property of the semigroup.

\begin{prop}
The semigroup $(T_p(t))_{t\geq 0}$ on $X_p$, $p\in [1,\infty)$, is
irreducible.
\end{prop}

\begin{proof}
It is enough to prove the statement for $p=2$, because the
irreducibility is inherited for the extrapolation semigroups in
$X_p,\, 1\leq p<\infty$, using Corollary \ref{comp} and
\cite[Theorem 7.2.2]{[Ar04]}. Since $X_2$ is reflexive and
$\left(T_2(t)\right)_{t\geq 0} $ is bounded, by~\cite[Example
V.4.7]{[EN00]} we have that the semigroup is \emph{mean ergodic}
(see~\cite[Definition V.4.3]{[EN00]}). By Theorem \ref{spectrum1}.1,
we obtain that the corresponding \emph{mean ergodic projection}
$P$ defined by
$$Px:=\lim_{r\to \infty} \frac{1}{r}\int_{0}^rT(s)xds$$
is the projection $\mathbf{1} \otimes \mathbf{1}$ onto the
subspace $<\!\!\mathbf{1}\!\!>$. Let $\{0\}\neq J\subset X_2$ be a
closed invariant ideal for $\left(T(t)\right)$, that is
$T(t)J\subset J,\, t\geq 0.$ Then also $PJ\subset J$ holds. By
definition, $PJ\subset <\!\!\mathbf{1}\!\!> $ and so $J$ contains
the closed ideal generated by a constant function -- hence the whole
space $X_2$. From this follows that the semigroup is irreducible.
\end{proof}

Knowing that our semigroup is irreducible we can now show its norm
convergence towards a projection of rank one.

\begin{cor}
For the semigroup $(T_p(t))_{t\geq 0}$ on $X_p$, $p\in [1,\infty)$,
the following hold.
\begin{enumerate}
\item The limit $Pf:=\lim_{t\to\infty}T_p(t)f$ exists for every
$f\in X_p$. \item $P$ is a strictly positive projection onto
$<\!\!\mathbf{1}\!\!>={\rm Ker} A$, the one-dimensional subspace
of $X_p$ spanned by the constant function ${\mathbf 1}$. \item For
every $\varepsilon > 0$ there exists  $M >0$ such that
\begin{equation}\label{convspeed}\Vert T_p(t)-P\Vert\le
Me^{(\varepsilon + \lambda_2) t} \qquad\hbox{for all }t\geq 0,\;
\end{equation} where $\lambda_2$ is the largest nonzero eigenvalue
of the generator $A$.
\end{enumerate} \end{cor}

\begin{proof}
The first assertions follows from~\cite[Cor.C-IV.2.10]{[Na86]} and
the second one from ~\cite[C-III.3.5.(d)]{[Na86]} and Theorem
\ref{spectrum1}. Since $P$ is the residue corresponding to the
spectral value $\lambda_1=0$ and $0$ is a first order pole of the
resolvent, estimate (3.2) in~\cite[Cor.V.3.2]{[EN00]} yields the third
assertion.
\end{proof}

Note that the property of converging to an equilibrium does not
depend on the structure of the network. However, the speed of the
convergence towards a projection is  determined by the second
largest eigenvalue $\lambda_2$ of $A$ and thus relies on the
network.

Combining graph theory and results about the spectrum of $A$,
obtained in Section \ref{secchareq},  we can now draw some further
estimates containing graph parameters for the speed of convergence
in \eqref{convspeed}. Let us demonstrate this for the case when all
$c_j=1$ and the graph is \emph{regular}, that is, every vertex has
the same degree. This means that
$$\left|\Gamma(\mv_i)\right|=\gamma \text{ for all }
i=1,\dots,n, \quad\text{and}\quad D=\gamma\cdot I. $$ Two generic
examples of regular graphs are the $n$\emph{-cycle} $C_n$ and the
\emph{complete graph} $K_n$ on $n$ vertices (in the latter, every two
vertices are connected by an edge).

In this case the characteristic equation becomes $$-\lambda \in
\sigma \left( A\right) \Longleftrightarrow
\det\left(\cos\sqrt{\lambda}\cdot \gamma \cdot I-\mathbb{A}\right)=0
\Longleftrightarrow \cos\sqrt{\lambda}\cdot \gamma \in
\sigma\left(\mathbb{A}\right),$$ see the proof of Theorem
\ref{spectrum2} or~\cite[Sec.~6]{[Be85]}. Using the Laplacian of the
graph, defined in \eqref{graphlaplace}, we obtain for its spectrum
$$\nu \in \sigma\left(\mathcal{L}\right) \Longleftrightarrow \det\left(\nu \cdot I- (\gamma \cdot I
-\ba)\right)=0 \Longleftrightarrow -\nu +\gamma \in \sigma(\ba).
$$ Hence, investigating the spectrum of the generator $A$, we are looking for $\lambda$'s such that
\begin{equation}\label{specaspecl}\lambda= -\left( 2l\pi \pm \mathrm{arc}\cos \left(1-\frac{\nu}{\gamma}\right) \right)^{2},\, l\in
\mathbb{Z}\end{equation}where $\nu \in \sigma \left(\mathcal{L}\right)$. The spectrum of $\mathcal{L}$ is sometimes also called the \emph{spectrum of the graph} and is well investigated in graph theory. In the following we always refer to the survey paper~\cite{[Mo91]}.

\begin{exa}\label{spCK}
In the case when $G=C_n$, the eigenvalues of $\mathcal{L}$ are precisely the numbers
$$\nu_k=2-2\cos\left(\frac{2(k-1)\pi}{n}\right),\quad k=1,\dots ,n, $$
while for $G=K_n$ we have
$$\nu_1=0 \text{ and } \nu_k = n \text{ for } 2\leq k\leq n.$$
\end{exa}

As explained above we are interested in the \emph{second-smallest} eigenvalue $\nu_2 \in
\sigma\left(\mathcal{L}\right)$, which is also called the \emph{algebraic connectivity} of the graph, see
\cite{[Mo91]}. It is related to the classical connectivity parameters of the graph (see below). If we look at
\eqref{specaspecl} we can easily conclude that
\begin{equation}\label{lambda2nu2}\lambda_2= -\left(
\mathrm{arc}\cos \left(1-\frac{\nu_2}{\gamma}\right) \right)^{2},\end{equation} since the function
$\mathrm{arc}\cos$ is strictly monotone decreasing and assumes its values between $0$ and $\pi$.% One can also see that the \emph{larger} $\nu_2$ the \emph{smaller} $\lambda_2$ is -- hence, in \eqref{convspeed}, the convergence towards the one-dimensional projection becomes faster (because $\lambda_2$ is negative).

\begin{exa} By increasing the number of vertices $n$, the convergence gets slower on the cycle $C_n$ and on the complete graph $K_n$. In fact, for $C_n$ we have $\gamma=2$, hence by Example \ref{spCK} we obtain that 
$$\lambda_2=-\left(\mathrm{arc}\cos \left(1-(1-\cos\frac{2\pi}{n})\right) \right)^{2}=-\frac{4\pi^2}{n^2}.$$ 
For $K_n$ we have $\gamma=n-1$, hence by
Example \ref{spCK}, $$\lambda_2=-\left( \mathrm{arc}\cos \left(1-\frac{n}{n-1}\right) \right)^{2}=-\left(
\mathrm{arc}\cos \left(-\frac{1}{n-1}\right) \right)^{2}.$$
\end{exa}

%We will now see what happens if we add new edges to $G$.
%
%\begin{exa}
%Let $G'=G+\me$ be the graph obtained from $G$ by inserting a new edge $\me$ into $G$, connecting two existing vertices. Here, we do not allow multiple edges. By \cite[Theorem 3.2]{[Mo91]}, we have $$\nu_2(G)\leq \nu_2(G').$$ This implies that for a given number of vertices the convergence speed among regular graphs is the biggest for the complete graph $K_n$.
%\end{exa}
%
%Let us now investigate what happens if instead of edges we add new \emph{vertices} (and some incident edges) to $G$. In this case, $G$ becomes a so-called \emph{spanning subgraph} of the new graph $G'$. By \cite[Corollary 3.4]{[Mo91]} we have
%$$\nu_2(G)\leq \nu_2(G').$$
%
%\begin{exa}
%Let $G'$ be a graph obtained from $G$ by adding some new vertices and new edges incident to these vertices such that the graph remains connected and regular. Then the semigroup belonging to $G'$ converges faster than the one belonging to $G$. This also implies that considering complete graphs, increasing the number of vertices increases the convergence speed.
%\end{exa}

If we have an estimate from \emph{below} for $\nu_2$, using \eqref{convspeed} and formula \eqref{lambda2nu2}, we
obtain an \emph{upper} estimate for the convergence speed of the semigroup to the one-dimensional projection. In
\cite{[Mo91]} we find many estimates for $\nu_2$ that use several graph parameters. As an example we mention the
next.

\begin{defi}
The \emph{edge connectivity parameter} $\eta=\eta(G)$ of the graph $G$ is defined as the minimum number of edges
whose deletion from
 $G$ disconnects it.
\end{defi}
\begin{exa}
By \cite[Theorem 6.2(b)]{[Mo91]},
$$\nu_2\geq 2 \eta(1-\cos \frac{\pi}{n})$$ holds, where $\eta$ is the edge connectivity parameter of
the graph. From this we obtain for \eqref{convspeed} that for every $\varepsilon > 0$ there exists an $M > 0$
such that
$$\Vert T_p(t)-P\Vert\le Me^{(\varepsilon - ( \mathrm{arc}\cos (1-\frac{2\eta}{\gamma}(1-\cos \frac{\pi}{n})
))^{2}) t} \qquad\hbox{for all }t\geq 0.$$
\end{exa}
Another estimate can be obtained using the \emph{diameter} of $G$.
\begin{defi}
The \emph{diameter} of $G$ denoted by $\mathrm{diam}(G)$ is the largest number of vertices which must be
traversed in order to travel from one vertex to another when paths which backtrack, detour, or loop are excluded
from consideration.
\end{defi}
\begin{exa}
By \cite[(6.10)]{[Mo91]}, we have $$\nu_2\geq \frac{4}{n\cdot \mathrm{diam}(G)}.$$ Hence, again we can conclude
that for every $\varepsilon > 0$ there exists an $M > 0$ such that
$$\Vert T_p(t)-P\Vert\le Me^{(\varepsilon - ( \mathrm{arc}\cos (1-\frac{4}{\gamma\cdot n\cdot \mathrm{diam}(G)})
)^{2}) t} \qquad\hbox{for all }t\geq 0.$$
\end{exa}

{\small $ $\\
Marjeta Kramar Fijav\v{z}\\
Faculty of Civil and Geodetic Engineering\\
University of Ljubljana\\
Jamova 2\\
SI-1000 Ljubljana\\
Slovenia\\
{\tt marjeta.kramar@fgg.uni-lj.si}

$ $\\
Delio Mugnolo\\
Abteilung Angewandte Analysis\\
Universit{\"a}t Ulm\\
Helmholtzstra{\ss}e 18\\
D-89081 Ulm\\
Germany\\
{\tt delio.mugnolo@uni-ulm.de}\\
\emph{and}\\
Dipartimento di Matematica\\
Universit\`{a} degli Studi di Bari\\
Via Orabona 4\\
I-70125 Bari\\
Italy

$ $\\
Eszter Sikolya\\
Department of Applied Analysis\\
E\"{o}tv\"{o}s Lor\'{a}nd University\\
P\'{a}zm\'{a}ny P\'{e}ter st. 1/c\\
H-1117 Budapest\\
Hungary\\
{\tt seszter@cs.elte.hu}}
\end{document}